\documentclass[psamsfonts,fceqn,leqno]{amsart}
\usepackage{mathrsfs,latexsym,amsfonts,amssymb,curves,epic}
\usepackage{color}

\newtheorem{theorem}{Theorem}[section]
\newtheorem{corollary}[theorem]{Corollary}
\newtheorem{proposition}[theorem]{Proposition}
\newtheorem{lemma}[theorem]{Lemma}
\newtheorem{question}[theorem]{Question}
\newtheorem{problem}[theorem]{Problem}
\theoremstyle{definition}

\newtheorem{remark}[theorem]{Remark}
\newtheorem{example}[theorem]{Example}

\begin{document}
\title[The characterizations of dense-pseudocompact and dense-connected spaces]
{The characterizations of Dense-pseudocompact and dense-connected spaces}

  \author{Fucai Lin}
  \address{(Fucai Lin): School of mathematics and statistics,
  Minnan Normal University, Zhangzhou 363000, P. R. China}
  \email{linfucai@mnnu.edu.cn; linfucai2008@aliyun.com}

  \author{Qiyun Wu}
  \address{(Qiyun Wu): School of mathematics and statistics,
  Minnan Normal University, Zhangzhou 363000, P. R. China}
\email{904993706@qq.com}

  \thanks{The first author is supported by the Key Program of the Natural Science Foundation of Fujian Province (No: 2020J02043), the NSFC (No. 11571158), the Institute of Meteorological Big Data-Digital Fujian and Fujian Key Laboratory of Data Science and Statistics.}

  %insert keywords
  %insert subject class

  %\date{\today}
  \begin{abstract}
Assume that $\mathcal{P}$ is a topological property of a space $X$, then we say that $X$ is {\it dense-$\mathcal{P}$} if each dense subset of $X$ has the property $\mathcal{P}$. In this paper, we mainly discuss dense subsets of a space $X$, and we prove that:

\smallskip
(1) if $X$ is Tychonoff space, then $X$ is dense-pseudocompact iff the range of each continuous real-valued function $f$ on $X$ is finite, iff $X$ is finite, iff $X$ is hereditarily pseudocompact;

\smallskip
(2) $X$ is dense-connected iff $\overline{U}=X$ for any non-empty open subset $U$ of $X$;

\smallskip
(3) $X$ is dense-ultraconnected iff for point $x\in X$, we have $\overline{\{x\}}=X$ or $\{x\}\cup (X\setminus\overline{\{x\}})$ is the unique open neighborhood of $x$ in $\{x\}\cup (X\setminus\overline{\{x\}})$, iff
for any two points $x$ and $y$ in $X$, we have $x\in \overline{\{y\}}$ or $y\in \overline{\{x\}}$.

Moreover, we give a characterization of a topological group (resp., paratopological group, quasi-topological group) $G$ such that $G$ is dense-connected.
  \end{abstract}
 \maketitle
\section{Introduction and terminology}
The dense subsets play an important role in the study of the theory of topological spaces, see \cite{AT2008, E1989}. In 1975, R. Levy and R.H. McDowell \cite{LM1975} introduced the concept of dense-separable, that is, each dense subset of a space is separable; then, in 1976, J.H. Weston and J. Shilleto \cite{WS1976} studied the cardinality of dense subsets of topological spaces. In 1989, I. Juh\'{a}sz and S. Shelah \cite{JS1989} proved that the $\pi$-weight of a compact Hausdorff space $X$ is equal to the supremum of the density of all dense subsets of $X$. Recently, A. Dow and I. Juh\'{a}sz \cite{DJ2020} has proved that each dense subset of a compact space can be covered by countably many compact subsets iff it has a countable $\pi$-weight; in \cite{LWL2023}, the authors systematically discuss the dense-separable spaces. By the definition of dense-separable, we can define the following concepts.

Let $X$ be a space and $\mathcal{P}$ be a topological property of $X$. We say that $X$ is {\it dense-$\mathcal{P}$} if each dense subset of $X$ has the property $\mathcal{P}$; $X$ is said to be {\it one-dense-$\mathcal{P}$} if there exists a dense subset $D$ of $X$ such that $D$ has the property $\mathcal{P}$; $X$ is said to be {\it proper one-dense-$\mathcal{P}$} if there exists a proper dense subset $D$ of $X$ such that $D$ has the property $\mathcal{P}$. Clearly, each proper one-dense-$\mathcal{P}$ is one-dense-$\mathcal{P}$, but not vice verse. We say that $X$ is {\it locally dense-$\mathcal{P}$} if for each point $x$ of $X$ and any neighborhood $U$ of $x$, there exists a neighborhood $V$ of $x$ such that $V\subset U$ and $V$ is dense-$\mathcal{P}$; $X$ is said to be {\it locally one-dense-$\mathcal{P}$} if for each point $x$ of $X$ and any neighborhood $U$ of $x$, there exists a neighborhood $V$ of $x$ such that $V\subset U$ and $V$ is one-dense-$\mathcal{P}$; $X$ is said to be {\it locally proper one-dense-$\mathcal{P}$} if for each point $x$ of $X$ and any neighborhood $U$ of $x$, there exists a neighborhood $V$ of $x$ such that $V\subset U$ and $V$ is proper one-dense-$\mathcal{P}$.

Let $G$ be a semitopological group and $\mathcal{P}$ be a topological property of $G$. We say that $G$ is {\it dense-subgroup-$\mathcal{P}$} if each dense subgroup of $G$ has the property $\mathcal{P}$; $G$ is said to be {\it one-dense-subgroup-$\mathcal{P}$} if there exists a dense subgroup $H$ of $G$ such that $H$ has the property $\mathcal{P}$; $G$ is said to be {\it proper one-dense-subgroup-$\mathcal{P}$} if there exists a proper dense subgroup $H$ of $G$ such that $H$ has the property $\mathcal{P}$. Clearly, each dense-subgroup-$\mathcal{P}$ with a non-trivial dense subgroup is proper one-dense-subgroup-$\mathcal{P}$, and each proper one-dense-subgroup-$\mathcal{P}$ is one-dense-subgroup-$\mathcal{P}$, but not vice verse.

It is well known that the pseudocompactness and the connectedness are two very important topological concepts, which have been applied to many fileds in mathematics. Hence it is meaningful to discuss the dense-pseudocompact and dense-connected spaces.

A {\it semitopological group} $G$ is a group $G$ with a topology such
that the product map of $G\times G$ into $G$ is separately
continuous. A {\it quasitopological group} $G$ is a group $G$ with a topology such
that $G$ is a semitopological group and the inverse map of $G$ onto itself
associating $x^{-1}$ with arbitrary $x\in G$ is continuous. A {\it
paratopological group} $G$ is a group $G$ with a topology such that
the product maps of $G \times G$ into $G$ is jointly continuous.
A {\it topological group} $G$ is a group $G$ with a
(Hausdorff) topology such that $G$ is a paratopological group and the inverse map of $G$ onto itself
associating $x^{-1}$ with arbitrary $x\in G$ is continuous.

This paper is organized as follows. In Section 2, we mainly discuss some properties of dense-$\mathcal{P}$-property.

Section 3 is dedicated to the study of dense-pseudocompact spaces. We prove that for a Tychonoff space $X$, it is dense-pseudocompact iff the range of each continuous real-valued function $f$ on $X$ is finite, iff $X$ is finite, iff $X$ is hereditarily pseudocompact.

In Section 4, we discuss some properties of dense-connected space. We prove that $X$ is dense-connected iff $\overline{U}=X$ for any non-empty open subset $U$ of $X$. Moreover, we prove that a topological group $G$ is dense-connected if and only if $G$ is a indiscrete topological group; a paratopological group $G$ is dense-connected if and only if $UV^{-1}=G$ for any non-empty open neighborhoods $U$ and $V$ of $e$; a quasitopological group $G$ is dense-connected if and only if $UV=G$ for any non-empty open neighborhoods $U$ and $V$ of $e$.

In Section 5, we mainly prove that $X$ is dense-ultraconnected iff for any point $x\in X$, we have $\overline{\{x\}}=X$ or $\{x\}\cup (X\setminus\overline{\{x\}})$ is the unique open neighborhood of $x$ in $\{x\}\cup (X\setminus\overline{\{x\}})$, iff
for any two points $x$ and $y$ in $X$, we have $x\in \overline{\{y\}}$ or $y\in \overline{\{x\}}$.

Denote the sets of real number, positive integers, the closed unit interval and all non-negative integers by $\mathbb{R}$,  $\mathbb{N}$, $I$ and $\omega$, respectively. Let $\mathbb{T}$ be unit circle with the usual topology. For undefined notation and terminology, the reader may refer to \cite{AT2008} and
  \cite{E1989}.

 \maketitle
\section{some properties of dense-$\mathcal{P}$-property}
In this section, we mainly discuss some topological properties of a space which is dense-$\mathcal{P}$. In particular, we give some topological properties $\mathcal{P}$ such that if $\mathcal{P}$ is dense-$\mathcal{P}$ then $\mathcal{P}$ is a hereditarily $\mathcal{P}$-property. Clearly, hereditarily $\mathcal{P}$-property is dense-$\mathcal{P}$, but not vice versa. Indeed, we have the following obvious theorem.

\begin{theorem}\label{t44}
Let $\mathcal{P}$ be a topological property which is closed heredity. If a space $X$ is dense-$\mathcal{P}$, then $X$ is hereditary $\mathcal{P}$-property.
\end{theorem}

\begin{proof}
Take any subspace $Y$ of $X$. Then $D=Y\cup (X\setminus \overline{Y})$ is dense in $X$, hence $D$ has $\mathcal{P}$-property. Since $Y$ is closed in $D$, it follows from our assumption that $Y$ has $\mathcal{P}$-property. Therefore, $X$ is hereditary $\mathcal{P}$.
\end{proof}

\begin{remark}
 It is well-known that normality, paracompactness, Lindel\"{o}f-property, countable compactness, strong paracompactness, countablee paracompactness, weak paracompactness, sequentially compact, realcompact, perfect normality, $k$-spaces are closed hereditary, hence if each dense subset of a space has one of these topological properties then it must be hereditary.
\end{remark}

\begin{proposition}\label{p566}
If $X$ is dense-$\mathcal{P}$, then each dense subspace is also dense-$\mathcal{P}$.
\end{proposition}

\begin{proposition}
If $X$ is dense-$\mathcal{P}$ and $\mathcal{P}$ is clopen hereditarily, then each open subspace is dense-$\mathcal{P}$.
\end{proposition}

\begin{proof}
Take any open subset $U$ of $X$ and any dense subset $D$ of $U$. Then $D\cup (X\setminus\overline{U})$ is dense in $X$. Since $X$ is dense-$\mathcal{P}$, it follows that $D\cup (X\setminus\overline{U})$ has dense-$\mathcal{P}$ by Proposition~\ref{p566}. Clearly, $D$ is clopen in $D\cup (X\setminus\overline{U})$, hence $D$ has the $\mathcal{P}$-property. Therefore, each open subspace is dense-$\mathcal{P}$.
\end{proof}

Let $X, Y$ be two spaces and $f: X\rightarrow Y$ a map. We say that $f$ is {\it almost open} (resp., {\it open}) if the interior of $f(U)$ is nonempty in $Y$ (resp., $f(U)$ is open in $Y$) for each nonempty open subset $U$ of $X$. The following proposition is obvious.

\begin{proposition}
Let $f: X\rightarrow Y$ be an almost-open map and $X$ have dense-$\mathcal{P}$. If the topological property $\mathcal{P}$ is preserved by almost-open map, then $Y$ is dense-$\mathcal{P}$.
\end{proposition}

\begin{example}
There exists one-dense-compact space $X$ such that $X$ is not proper one-dense-pseudocompact.
\end{example}

\begin{proof}
Let $X$ be the one-point compactification of infinite discrete space $D$. Clearly, $X$ is compact. However, any proper dense subset is not pseudocompact.
\end{proof}

 \maketitle
\section{dense subsets of pseudocompact spaces}
In this section, we shall prove that each dense-pseudocompact space is finite, hence it is proper one-dense-pseudocompact and one-dense-pseudocompact, but not vice verse, see Remark~\ref{rem4}; moreover, each proper one-dense-pseudocompact is one-dense-pseudocompact, but not vice verse, see Remark~\ref{rem4}.

Clearly, we have the following proposition.

\begin{proposition}
Each one-dense-pseudocompact space is pseudocompact.
\end{proposition}

In order to prove our main theorem in this section, we need a lemma.

\begin{lemma}\label{p11}
A space $X$ is a dense-pseudocompact if and only if each open subspace of $X$ is dense-pseudocompact.
\end{lemma}

\begin{proof}
Sufficiency is obvious. Now it suffices to prove the necessity. Take any open subspace $U$ of $X$ and any dense subset $D$ of $U$. Then $Y=D\cup(X\setminus \overline{U})$ is dense in $X$, hence $Y$ is pseudocompact. Since $U\cap Y=D$ and $U$ is open in $X$, it follows that $D$ is open and closed in $Y$, hence $D$ is pseudocompact. Therefore, $U$ is dense-pseudocompact.
\end{proof}

Now we can give a characterization of a dense-pseudocompact space; indeed, it is a finite space.

\begin{theorem}\label{t4}
Let $X$ be a Tychonoff space. Then the following statements are equivalent:
\begin{enumerate}
\smallskip
\item $X$ is dense-pseudocompact;

\smallskip
\item the range of each continuous real-valued function $f$ on $X$ is finite;

\smallskip
\item $X$ is finite;

\smallskip
\item $X$ is hereditarily pseudocompact.
\end{enumerate}
\end{theorem}

\begin{proof}
Clearly, it suffice to prove that (1) $\Rightarrow$ (2) and (2) $\Rightarrow$ (3).

(1) $\Rightarrow$ (2). Suppose not, there exists a continuous real-valued function $h$ on $X$ such that $h(X)$ is an infinite subset of $\mathbb{R}$. Since $X$ is pseudocompact, $h(X)$ is bounded in $\mathbb{R}$, hence there exists a limit point $r\in h(X)$. Without loss of generality, we may assume that there exists a sequence $\{r_{n}: n\in\mathbb{N}\}$ such that $r_{n}<r$ and $r_{n}\rightarrow r$ as $n\rightarrow\infty$. Put $A=h^{-1}(r)$ and $U=X\setminus A$. Then $U$ is open in $X$. Since $X$ is dense-pseudocompact, it follows from Lemma~\ref{p11} that $U$ is dense-pseudocompact. Let $g: h(X)\setminus\{r\}\rightarrow\mathbb{R}$ be a function such that $g(x)=\tan(\frac{2x-r}{2r}\pi)$ for each $x\in (\frac{r}{2}, \frac{3r}{2})\cap (h(X)\setminus\{r\})$ and $g(x)=0$ for any $(h(X)\setminus\{r\})\setminus(\frac{r}{2}, \frac{3r}{2})$. Then $g(x)$ is an unbounded continuous function. Put $h_{1}=h|_{U}$ and $h_{2}=g\circ h_{1}$. Then $h_{2}$ is an unbounded continuous function from $U$ to $\mathbb{R}$, which is a contradiction with $U$ being pseudocompactness.

(2) $\Rightarrow$ (3). Assume $X$ is infinite. Then it is obvious that $X$ is not connected. Next, by induction, we shall construct a continuous real-valued function $f$ on $X$ such that the range of $f$ is infinite, which leads to a contradiction. Indeed, since $X$ is not connected, there exist two disjoint open subsets $U_{0}$ and $V_{0}$ such that $U_{0}\cup V_{0}=X$ and $U_{0}$ is infinite. Assume that we have defined open subsets $U_{i}$ and $V_{i}$ ($i\leq n$) such that the following conditions hold:

\smallskip
(1) $U_{0}\cup V_{0}=X$ and $U_{0}$ is infinite;

\smallskip
(2) for each $1\leq i\leq n$, we have $U_{i}\cup V_{i}=U_{i-1}$;

\smallskip
(3) for each $i\leq n$, $U_{i}\cap V_{i}=\emptyset$ and $U_{i}$ is infinite.

It is obvious that the range of each continuous real-valued function on $U_{n}$ is finite, hence $U_{n}$ is not connected, then there exists two disjoint open subsets $U_{n+1}$ and $V_{n+1}$ such that $U_{n+1}\cup V_{n+1}=U_{n}$ and $U_{n+1}$ is infinite. Put $A=X\setminus \bigcup_{i=0}^{\infty}V_{i}$. Now we have two sequences of open subsets $\{U_{n}: n\in\mathbb{N}\}$ and $\{V_{n}: n\in\mathbb{N}\}$ satisfy the conditions (1)-(3) above. Next we define a real-valued continuous function $f: X\rightarrow \mathbb{R}$ on $X$ as follows: for each $x\in X$, if $x\in A$, then $f(x)=0$; otherwise, there exists an unique $i\in\mathbb{N}$ such that $x\in V_{i}$, then $f(x)=\frac{1}{i+1}$. Clearly, $f$ is continuous and the range of $f$ is infinite.

Therefore, $X$ is finite.
\end{proof}

\begin{corollary}
If $X$ is dense-pseudocompact with $|X|>1$, then $X$ is not connected.
\end{corollary}

By Theorem~\ref{t4}, we have the following corollary.

\begin{corollary}
If $X$ is locally dense-pseudocompact, then $X$ is a discrete space.
\end{corollary}

However, the following question is still unknown for us.

\begin{question}\label{q0}
Let $G$ be a dense-subgroup-pseudocompact topological group. Is $G$ finite?
\end{question}

Next we give some partial answers to Question~\ref{q0}.

\begin{proposition}
Let $G$ be a separable infinite group. Then $G$ is not dense-subgroup-pseudocompact.
\end{proposition}

\begin{proof}
Assume $G$ is dense-subgroup-pseudocompact. Since $G$ is separable, it follows that there exists a countable dense subgroup $H$ of $G$, then $H$ is pseudocompact by our assumption. Thus $H$ is compact. However, each compact infinite group with a cardinality at least $\mathfrak{c}$, whcih is a contradiction.
\end{proof}

The density $d(X)$ of a space $X$, which is defined as the smallest cardinal number of the form $|A|$ for each dense subset $A$ of $X$.

\begin{corollary}
If $G$ is dense-subgroup-pseudocompact, then either $G$ is finite or $d(G)>\omega$.
\end{corollary}

Let $(X, \tau)$ be a topological space. The $P$-space topology on $X$
determined by $\tau$ is the smallest topology $P\tau$ on $X$ such that $\tau\subset P\tau$ and every
$G_{\delta}$-subset of $(X, P\tau)$ is $P\tau$-open. The set $X$ with the $P\tau$ topology is denoted $PX$.

From \cite[Theorem 2.9]{CM1989}, it follows that the following proposition holds.

\begin{proposition}
Let $G$ be a pseudocompact topological group. Then $G$ is dense-subgroup-pseudocompact if and only if each dense subgroup $H$ of $G$ is dense in $PG$.
\end{proposition}

\begin{remark}\label{rem4}
The class of proper one-dense-subgroup-pseudocompact groups was extensive studied, see \cite{CR1982,CR1988,D2019}; in particular, it follows from \cite[Theorem 1.1]{CM2007} that each pseudocompact abelian group $G$ of uncountable weight has a proper dense pseudocompact subgroup, thus $G$ is proper one-dense-subgroup-pseudocompact.
Obviously, each meitrziable pseudocompact infinite group $G$ is one-dense-subgroup-pseudocompact (thus one-dense-pseudocompact), which is not proper one-dense-pseudocompact, hence $G$ is not proper one-dense-subgroup-pseudocompact. Moreover, it follows from \cite[Theorem 1.1]{CM2007} that each pseudocompact, separable, non-metrizble topological group $G$ is proper one-dense-subgroup-pseudocompact (thus proper one-dense-pseudocompact, but it is not dense-pseudocompact by Theorem~\ref{t4}), such as $\mathbb{T}^{\mathfrak{c}}$. However $G$ is not dense-subgroup-pseudocompact; otherwise, there exists a countable pseudocompact dense-subgroup $H$, which implies that $H$ is compact, a contradiction.
\end{remark}

 \maketitle
\section{dense subsets of connected spaces}

In this section, we shall give some partial answers to Problem~\ref{q1}. Indeed, we prove that $G$ is not dense-connected for any non-indiscrete topological group $G$, see Corollary~\ref{c1}. Moreover, we give a characterization of a topological group (resp., paratopological group, quasi-topological group) $G$ such that $G$ is dense-connected.

The following proposition is obvious.

\begin{proposition}
Each one-dense-connected space is connected.
\end{proposition}

Moreover, it is obvious that each dense-connected semitopological group is dense-subgroup-connected. However, the following problems are interesting.

\begin{problem}
Is each dense-subgroup-connected topological group dense-connected?
\end{problem}

\begin{problem}\cite[Open Problem 1.4.3]{AT2008}\label{q1}
Does there exist an infinite (abelian, Boolean) topological group $G$ such that $G$ is dense-subgroup-connected?
\end{problem}

Clearly, $\mathbb{T}$ is one-dense-connected, but it is not proper one-dense-connected. In \cite{W1971}, the author proved that every compact connected Abelian group $G$ except $\mathbb{T}$ contains
a dense proper connected subgroup provided that $w(G)<\mathfrak{c}$ or $w(G)^{\aleph_{0}}=w(G)$, thus $G$ is proper one-dense-subgroup-connected, where $w(G)$ is the weight of $G$; moreover, in \cite{M1986}, the author proved that there exists a proper one-dense connected subgroup $G$ of $\mathbb{R}^{2}$, thus $\mathbb{R}^{2}$ is proper one-dense-subgroup-connected and not dense-subgroup-connected.

For an abelian $G$, we say that $G$ is an {\it $M$-group} if for every integer $m\geq 1$, either $|mG|=1$ or $|mG|\geq\mathfrak{c}$. By \cite[Corollary 1.10]{DS2016}, we have the following proposition.

\begin{proposition}\label{p9}
Let $G$ be a dense-subgroup-connected abelian group. Then each dense subgroup $H$ of $G$ is an $M$-group.
\end{proposition}

The following Theorem~\ref{t1} gives a partial answer to Problem~\ref{q1}. First, we give a concept.

A space with no disjoint open sets will be called {\it hyperconnected}. Clearly, each dense subspace of a hyperconnected space is hyperconnected, hence a space is hyperconnected if and only if it is dense-hyperconnected.

\begin{theorem}\label{t1}
Let $X$ be a space. Then the following statements are equivalent:
\begin{enumerate}
\smallskip
\item $X$ is dense-connected;

\smallskip
\item $\overline{U}=X$ for any non-empty open subset $U$ of $X$;

\smallskip
\item $X$ is hyperconnected.

\smallskip
\item $X$ is dense-hyperconnected.
\end{enumerate}
\end{theorem}

\begin{proof}
It suffice to prove that (1) $\Rightarrow$ (2), (2) $\Rightarrow$ (3) and (3) $\Rightarrow$ (1).

(1) $\Rightarrow$ (2). Let $X$ be dense-connected. Take an any non-empty open subset $U$ of $X$. Then $D=U\cup (X\setminus \overline{U})$ is dense in $X$, hence $D$ is dense-connected by our assumption. Moreover, $U$ is non-empty, open and closed in $D$, then $X\setminus \overline{U}=\emptyset$ since $D$ is connected. Hence $\overline{U}=X$.

(2) $\Rightarrow$ (3). Let $\overline{W}=X$ for any non-empty open subset $W$ of $X$. Take any non-empty open subsets $U$ and $V$ of $X$. Assume that $U\cap V=\emptyset$, then $U\cap \overline{V}=\emptyset$, which is contradiction with $\overline{V}=X$ by our assumption.

(3) $\Rightarrow$ (1). Let $X$ be hyperconnected. Then $U\cap V\neq\emptyset$ for any non-empty open subsets $U$ and $V$ of $X$. Take any dense subset $D$ of $X$. We claim that $D$ is connected. Suppose not, there exist proper open subsets $U_{1}$ and $V_{1}$ of $D$ such that $U_{1}\cap V_{1}=\emptyset$ and $U_{1}\cup V_{1}=D$. Then there exist two open subsets $U$ and $V$ in $X$ such that $U_{1}=U\cap D$ and $V_{1}=V\cap D$. By our assumption, $U\cap V$ is a non-empty open subset in $X$, hence $U\cap V\cap D\neq\emptyset$ since $D$ is dense in $X$. Then $U_{1}\cap V_{1}=(U\cap D)\cap (V\cap D)=U\cap V\cap D\neq\emptyset$, which is a contradiction.
\end{proof}

The following corollary shows that the dense-connected topological group is a indiscrete space.

\begin{corollary}\label{c1}
Let $G$ be a topological group. Then $G$ is dense-connected if and only if $G$ is a indiscrete topological group.
\end{corollary}

\begin{proof}
It suffices to prove the necessity. Assume $G$ is dense-connected. Assume $G$ is not indiscrete, then there exists an proper open neighborhood $U$ of $e$, hence we can find a symmetric open neighborhood $V$ of $e$ such that $V^{2}\subset U$ by the jointly continuous and inverse continuous. It is easy to see that $\overline{V}\subset U$. However, it follows from Theorem~\ref{t1} that $\overline{V}=G$, thus $U=G$, which is a contradiction.
\end{proof}

\begin{corollary}\label{c0}
Let $X$ be a dense-connected space with $|X|>1$. Then $X$ is not Hausdorff.
\end{corollary}

By Theorem~\ref{t1}, we have the following three propositions.

\begin{proposition}\label{p5}
If $X$ is a dense-connected subspace of $Y$, then any subset $Z$ with $X\subset Z\subset \overline{X}$ is dense-connected.
\end{proposition}

\begin{proposition}\label{p9}
If $X$ is dense-connected, then each open subspace is dense-connected.
\end{proposition}

\begin{proposition}\label{p0}
Let $\tau$ and $\sigma$ be two topologies on the set $X$ such that $\tau\subset\sigma$. If $(X, \sigma)$ is dense-connected, then $(X, \tau)$ is dense-connected.
\end{proposition}

The following example shows that there exists a $T_{1}$-compact dense-connected space.

\begin{example}\label{e2}
Let $X$ be an infinite countable set endowed with the finite complementary topology. Then $X$ is a dense-connected, compact and $T_{1}$-space.
\end{example}

However, the situation is quite different in the class of paratopological groups, see Example~\ref{e0}. Indeed, we also have the following two theorems.

\begin{theorem}\label{t2}
A paratopological group $G$ is dense-connected if and only if $UV^{-1}=G$ for any non-empty open neighborhoods $U$ and $V$ of $e$.
\end{theorem}

\begin{proof}
Assume that $G$ is dense-connected, and assume that $\mathcal{N}(e)$ is the family of all open neighborhoods of $e$. Take any non-empty open neighborhoods $U$ and $V$ of $e$. It is easy to see that $\overline{U}=\bigcap_{W\in\mathcal{N}(e)}UW^{-1}$, thus $\overline{U}\subset UV^{-1}$. By (2) of Theorem~\ref{t1}, we have $\overline{U}=G$, hence $UV^{-1}=G$.

Conversely, assume that $UV^{-1}=G$ for any non-empty open neighborhoods $U$ and $V$ of $e$.  By Theorem~\ref{t1}, it suffices to prove that $\overline{W}=G$ for any nonempty open subset $W$ of $G$. Pick any point $w\in W$ and put $O=w^{-1}W$. Then $O$ is an open neighborhood of $e$, hence $\overline{O}=\bigcap_{W\in\mathcal{N}(e)}OW^{-1}=G$ by our assumption, that is, $\overline{O}=G$. Therefore, $G=w\overline{O}=w\overline{w^{-1}W}=\overline{ww^{-1}W}=\overline{W}$.
\end{proof}

By \cite[Proposition 1.4.5]{AT2008}, we can prove the following theorem by a similar proof of Theorem~\ref{t2}.

\begin{theorem}\label{t3}
A quasitopological group $G$ is dense-connected if and only if $UV=G$ for any non-empty open neighborhoods $U$ and $V$ of $e$.
\end{theorem}

\begin{example}\label{e0}
There exists an infinite, $T_{1}$ and non-indiscrete paratopological group $G$ such that $G$ is dense-connected.
\end{example}

\begin{proof}
Indeed, let $G=(\mathbb{R}, +)$ endowed with a topology which has a base as the following family $$\tau=\{x+[y, +\infty)| x, y\in\mathbb{R}\}.$$ Then $(G, \tau)$ is a $T_{1}$ and infinite non-indiscrete paratopological group. By Theorem~\ref{t2}, $G$ is dense-connected.
\end{proof}

\begin{example}\label{e1}
There exists an infinite, $T_{1}$ and non-indiscrete quasitopological group $G$ such that $G$ is dense-connected.
\end{example}

\begin{proof}
Let $G=(\mathbb{Z}, +)$ endowed with a topology which has a base as the following family $$\tau=\{x+\{y: y\in\mathbb{Z}, |y|\geq n\}: x\in\mathbb{Z}, n\in\mathbb{N}\}.$$Then $(G, \tau)$ is an infinite, $T_{1}$ and non-indiscrete quasitopological group. By Theorem~\ref{t3}, $G$ is dense-connected.
\end{proof}

\begin{remark}\label{rem3}
Moreover, it follows from Proposition~\ref{p0} that there exists an infinite, $T_{0}$ and dense-connected paratopological group $G$ which is not $T_{1}$. Indeed, let $G=(\mathbb{R}, +)$ endowed with the following topology $$\delta=\{x+[0, +\infty)| x\in\mathbb{R}\}.$$ Clearly, $(G, \delta)$ is a $T_{0}$ and infinite non-indiscrete paratopological group which is not $T_{1}$. Since $\delta\subset\tau$, where $\tau$ is the topology in Example~\ref{e0}, it follows from Proposition~\ref{p0} that $(G, \delta)$ is dense-connected.
\end{remark}

A space $X$ is said to be {\it Brown} \cite{B2022} if for any nonempty open subsets $U$ and $V$ of $X$ the intersection $\overline{U}\cap \overline{V}$ is infinite. Clearly, each dense-connected infinite space is Brown. However, there exists a Brown space $X$ such that $X$ is not dense-connected. Indeed, there exists a countable Haudorff Brown space $X$, see \cite{B1953}; hence $X$ is not dense-connected by Corollary~\ref{c0}.

The following proposition shows that dense-connectedness is preserved by continuous onto map.

\begin{proposition}\label{p1}
Let $f: X\rightarrow Y$ be a continuous onto map. If $X$ is dense-connected, then $Y$ is dense-connected.
\end{proposition}

\begin{proof}
By Theorem~\ref{t1}, it suffices to prove that $\overline{U}=Y$ for any non-empty open subset $U$ of $Y$. Take any nonempty open subset $U$ of $Y$; then $f^{-1}(U)$ is open in $X$, hence $\overline{f^{-1}(U)}=X$ since $X$ is dense-connected. From the continuity of $f$, it follows that $X=\overline{f^{-1}(U)}\subset f^{-1}(\overline{U})$, thus $Y=f(X)\subset f(f^{-1}(\overline{U}))$, that is, $\overline{U}=Y$. Therefore, $Y$ is dense-connected.
\end{proof}

Next we discuss some property of the dense-connectedness.

\begin{proposition}\label{p2}
Let $\{X_{\alpha}\}_{\alpha\in I}$ be a family of spaces. Then $\Pi_{\alpha\in I}X_{\alpha}$ is dense-connected if and only if each $X_{\alpha}$ is dense-connected.
\end{proposition}

\begin{proof}
By Proposition~\ref{p1}, the necessity is obvious. Now we prove the sufficiency . By Theorem~\ref{t1}, it suffices to prove that $\overline{U}=\Pi_{\alpha\in I}X_{\alpha}$ for any nonempty open subset $U$ of $\Pi_{\alpha\in I}X_{\alpha}$. Fix any nonempty open subset $U$ of $\Pi_{\alpha\in I}X_{\alpha}$; then there exists a basic open subset $B=\Pi_{\alpha\in F}U_{\alpha}\times \Pi_{\alpha\in I\setminus F}X_{\alpha}$ such that $B\subset U$, where $F$ is a finite subset of $I$ and each $U_{\alpha}$ is a nonempty subset of $X_{\alpha}$. From Theorem~\ref{t1}, it follows that $\Pi_{\alpha\in I}X_{\alpha}=\Pi_{\alpha\in F}\overline{U_{\alpha}}\times \Pi_{\alpha\in I\setminus F}X_{\alpha}\subset \overline{B}\subset \overline{U}$, hence $\overline{U}=\Pi_{\alpha\in I}X_{\alpha}$.
\end{proof}

\begin{proposition}
Let $\{X_{\alpha}\}_{\alpha\in I}$ be a family of dense-connected subspaces of a space $X$. If, for any nonempty open subset $U$ of $X$, we have $U\cap (\bigcup_{\alpha\in I}X_{\alpha})=\emptyset$ or $U\cap X_{\alpha}\neq\emptyset$ for each $\alpha\in I$, then $\bigcup_{\alpha\in I}X_{\alpha}$ is dense-connected.
\end{proposition}

\begin{proof}
Let $Y=\bigcup_{\alpha\in I}X_{\alpha}$. By Theorem~\ref{t1}, it suffices to prove that $\overline{W}\cap Y=Y$ for any nonempty open subset $W$ of $Y$. Take an arbitrary nonempty open subset $W$ of $Y$. Then there exists an open subset $U$ of $X$ such that $W=U\cap Y$. By our assumption and $W\subset Y$, we have $U\cap X_{\alpha}\neq\emptyset$ for each $\alpha\in I$, hence $X_{\alpha}\subset(\overline{U\cap X_{\alpha}})\cap Y$ since $X_{\alpha}$ is dense-connected for each $\alpha\in I$. Therefore, it follows that $$Y=\bigcup_{\alpha\in I}X_{\alpha}\subset\bigcup_{\alpha\in I}(\overline{U\cap X_{\alpha}}\cap Y)\subset\overline{U\cap Y}\cap Y=\overline{W}\cap Y$$
\end{proof}

\begin{proposition}\label{p4}
Let $\{X_{\alpha}\}_{\alpha\in I}$ be a family of dense-connected open subspaces of a space $X$. If $X_{\alpha}\cap X_{\beta}\neq\emptyset$ for any two distinct $\alpha, \beta\in I$, then $\bigcup_{\alpha\in I}X_{\alpha}$ is dense-connected.
\end{proposition}

\begin{proof}
Let $Y=\bigcup_{\alpha\in I}X_{\alpha}$. Clealry, $Y$ is open in $X$. By Theorem~\ref{t1}, it suffices to prove that $\overline{W}\cap Y=Y$ for any nonempty open subset $W$ of $Y$. Take an arbitrary nonempty open subset $W$ of $Y$. Then $W$ is open in $X$. By our assumption and $W\subset Y$, we have $W\cap X_{\alpha_{0}}\neq\emptyset$ for some $\alpha_{0}\in I$, hence $X_{\alpha_{0}}\subset(\overline{W\cap X_{\alpha_{0}}})\cap Y$ since $X_{\alpha_{0}}$ is dense-connected. Therefore, for each $\alpha\in I$, we have $\emptyset\neq X_{\alpha}\cap X_{\alpha_{0}}\subset (\overline{W\cap X_{\alpha_{0}}})\cap Y$; since $X_{\alpha}\cap X_{\alpha_{0}}$ is a nonempty open subset of $X_{\alpha}$, it follows that $$X_{\alpha}\subset(\overline{X_{\alpha}\cap X_{\alpha_{0}}})\cap Y\subset (\overline{W\cap X_{\alpha_{0}}})\cap Y\subset \overline{W}\cap Y.$$ Therefore, $Y=\overline{W}\cap Y$.
\end{proof}

Finally, we discuss the locally dense-connected spaces. By Theorem~\ref{t1}, each dense-connected space is locally dense-connected; moreover, there exists a locally dense-connected space which is not connected, such as, discrete space with at least two points.

Let $X$ be a space and $x\in X$. Then the maximal dense-connected subset, which containing the point $x$, is called the {\it dense-connected component of point $x$}; we denote the dense-connected component of $x$ by $DC(x)$.

\begin{proposition}\label{p6}
Let $X$ be a locally dense-connected space. Then each $DC(x)$ is open and closed.
\end{proposition}

\begin{proof}
Take any $x\in X$. Let $$\mathcal{U}=\{U: x\in U\ \mbox{and}\ U\ \mbox{is an open dense-connected subset in}\ X\}$$ and put $W=\bigcup\mathcal{U}$. Since $X$ is locally dense-connected, $W$ is a nonempty open set. From Proposition~\ref{p4}, it follows that $W$ is dense-connected. We claim that $W$ is closed. Suppose not, there exists $w\in \overline{W}\setminus W$. Since $X$ is locally dense-connected, there exists open dense-connected neighborhood $O$ of $w$, then $O\cap W\neq\emptyset$. By Proposition~\ref{p4}, $O\cup W$ is dense-connected, hence $O\cup W\subset W$, which is a contradiction. Then it is easy to see that $W$ is the dense-connected component of $x$.
\end{proof}

By Propositions~\ref{p4} and~\ref{p6}, we have the following proposition.

\begin{proposition}\label{p7}
Let $X$ be a locally dense-connected space. For any two dense-connected components $DC(x)$ and $DC(y)$, we have $DC(x)=DC(y)$ or $DC(x)\cap DC(y)=\emptyset$.
\end{proposition}

By Propositions~\ref{p9} and~\ref{p6}, it is easily verified that the following result holds.

\begin{proposition}\label{p10}
A space $X$ is locally dense-connected if and only if the dense-connected components of all open subspace of $X$ are open.
\end{proposition}

By Proposition~\ref{p10}, the following theorem holds.

\begin{theorem}\label{t233}
A space $X$ is locally dense-connected if and only if $X$ is the topological sum of a family dense-connected spaces.
\end{theorem}

The following Proposition~\ref{p3} is easily verified.

\begin{proposition}\label{p3}
Local dense-connectedness is an invariant of open maps.
\end{proposition}

However, the following question is unknown for us.

\begin{question}
Is local dense-connectedness an invariant of quotient maps?
\end{question}

By Proposition~\ref{p3}, we can prove the following proposition by a similar proof of Proposition~\ref{p2}.

\begin{proposition}
Let $\{X_{\alpha}\}_{\alpha\in I}$ be a family of spaces. Then $\Pi_{\alpha\in I}X_{\alpha}$ is locally dense-connected if and only if there exists a finite $F\subset I$ such that $X_{\alpha}$ is locally dense-connected for each $\alpha\in F$ and $X_{\alpha}$ is dense-connected for each $\alpha\in I\setminus F$.
\end{proposition}

\begin{question}
Is each locally compact connected group $G$ with uncountable weight proper one-dense-subgroup-connected?
\end{question}

\begin{theorem}
If $G$ is a connected, pseudocompact abelian group with $\omega<w(G)\leq\mathfrak{c}$, then $G$ is proper one-dense-subgroup-connected.
\end{theorem}

\begin{proof}
From \cite[Theorem 4.3]{CM1989}, it follows that there exists a proper one-dense, pseudocompact subgroup $H$ of $G$, then from \cite[Corollary 2.3]{CM1989},  we conclude that $H$ is connected.
\end{proof}

 \maketitle
\section{dense subsets of pathwise connected spaces and ultraconnected spaces}
In this section, we mainly discuss the dense subsets of pathwise connected spaces and ultraconnected spaces, and give some characterizations of dense-untraconnected spaces. First, we consider the following question.

\begin{question}\label{q2}
How to characterise a space $X$ such that $X$ is (locally) dense-pathwise connected?
\end{question}

Now we give some partial answers to Question~\ref{q2}.

\begin{proposition}\label{l4}
Let $X$ be the finite complement topology. Then $X$ is not dense-pathwise connected.
\end{proposition}

\begin{proof}
If $X$ is finite, then it is obvious. Assume $X$ is infinite, then there exists a countable infinite subset $Y$ which is dense in $X$. Clearly, $Y$ is not pathwise connected. Therefore, $X$ is not dense-pathwise connected.
\end{proof}

The following two results are obvious.

\begin{proposition}
Let $X$ be an uncountable set endowed with the finite complement topology. Then $X$ is proper one-dense-pathwise connected.
\end{proposition}

\begin{lemma}\label{l3}
Let $(X, \tau)$ be a pathwise connected space. If $\delta$ is a topology on $X$ such that $\delta\subset\tau$, then $(X, \delta)$ is also pathwise connected.
\end{lemma}

\begin{proposition}
If $X$ is $T_{1}$ and dense-pathwise connected, then $X$ is not separable.
\end{proposition}

\begin{proof}
Assume $X$ is separable, then there exists a countable infinite subset $Y$ which is dense in $X$. Then $Y$ is pathwise connected. However, Example~\ref{e2} shows that any finite complement of countable infinite space, which is the coarsest $T_{1}$-topology, is not pathwise connected, which leads to a contradiction by Lemma~\ref{l3}.
\end{proof}

The answer to the following question is interesting.

\begin{question}\label{q4}
Does there exist a $T_{1}$-space $X$ such that $X$ is dense-pathwise connected?
\end{question}

The following Theorem~\ref{t5555} gives a partial answer to Question~\ref{q4}.

Let $X$ be a space and $x, y$ two distinct points of $X$. We say that $x$ and $y$ are not {\it $T_{1}$-points} if there exists a point $z\in\{x, y\}$ such that $\{x, y\}\subset U$ for each open neighborhood $U$ of $z$ in $X$. Moreover, we say that $X$ is a {\it non-separated-points space}. The following two lemmas are easy.

\begin{lemma}\label{l1}
Let $X$ be a space and $x, y$ two distinct points of $X$. If $x$ and $y$ are not $T_{1}$-points, then the map $f: \mathbb{I}\rightarrow X$, defined by $f([0, 1))=\{x, y\}\setminus\{z\}$ and $f(1)=z$, is a continuous map.
\end{lemma}

\begin{lemma}\label{l2}
If $X$ is a non-separated points space, then each subspace of $X$ is a non-separated point subspace.
\end{lemma}

By Lemmas~\ref{l1} and~\ref{l2}, it is easily verified that the following theorem holds.

\begin{theorem}\label{t5555}
If $X$ is a non-separated points space, then $X$ is dense-pathwise connected.
\end{theorem}

\begin{corollary}
If $G$ is a non-$T_{1}$ semitopological group, then $G$ is dense-pathwise connected.
\end{corollary}

\begin{remark}
(i) Obviously, each dense-pathwise connected is proper one-dense-pathwise connected, and each proper one-dense-pathwise connected is one-dense-pathwise connected, but not vice verse. Clearly, Example~\ref{e2} is a dense-connected space which is not pathwise connected; indeed, it is not one-dense-pathwise connected.

\smallskip
(ii) Each pathwise connected space is one-dense-pathwise connected. In \cite{DS2016}, the concept of one-dense-pathwise connected is called densely pathwise connected, which plays an important role in the study of \cite{DS2016}, where they proved that every abelian $M$-group of infinite exponent admits an one-dense-pathwise connected, locally one-dense-pathwise connected group topology.

\smallskip
(iii) Since each pathwise connected space is connected, it follows from Corollary~\ref{c0} that each dense-pathwise connected space is not Hausdorff.

\smallskip
(iv) Clearly, each nontrivial discrete space is locally dense-pathwise connected and not dense-pathwise connected; from Theorem~\ref{t1}, it follows that each dense-pathwise connected space is locally dense-pathwise connected. Moreover, we can prove that a space $X$ is locally dense-pathwise connected if and only if $X$ is the topological sum of dense-pathwise connected spaces by a similar proof of Theorem~\ref{t233}.
\end{remark}

Finally, we give a characterization of a space $X$ such that $X$ is dense-ultraconnected. A space $X$ is said to be {\it ultraconnected} if it has no disjoint closed subsets of $X$. Clearly, each ultraconnected space is connect.

\begin{theorem}\label{t22}
Let $X$ be a space. Then the following statements are equivalent:
 \begin{enumerate}
\smallskip
\item $X$ is dense-ultraconnected;

\smallskip
\item for point $x\in X$, we have $\overline{\{x\}}=X$ or $\{x\}\cup (X\setminus\overline{\{x\}})$ is the unique open neighborhood of $x$ in $\{x\}\cup (X\setminus\overline{\{x\}})$;

\smallskip
\item for any two points $x$ and $y$ in $X$, we have $x\in \overline{\{y\}}$ or $y\in \overline{\{x\}}$.
\end{enumerate}
\end{theorem}

\begin{proof}

(3) $\Rightarrow$ (1). It is obvious. We only need to prove (1) $\Rightarrow$ (2) and (2) $\Rightarrow$ (3).

(1) $\Rightarrow$ (2). Assume $X$ is dense-ultraconnected. Take any point $x\in X$. Assume that $\overline{\{x\}}\neq X$, then $X\setminus\overline{\{x\}}$ is a nonempty open set in $X$. Put $Y=\{x\}\cup (X\setminus \overline{\{x\}})$. Then $Y$ is dense in $X$, hence $Y$ is ultraconnected and $\{x\}$ is closed in $Y$. Assume that there exists an open neighborhood $U$ of $x$ in $Y$ such that $U\neq Y$, then $Y\setminus U$ is nonempty and closed in $Y$ and $(Y\setminus U)\cap \{x\}=\emptyset$, which is a contradiction. Therefore, $\{x\}\cup (X\setminus\overline{\{x\}})$ is the unique open neighborhood of $x$ in $\{x\}\cup (X\setminus\overline{\{x\}})$.

(2) $\Rightarrow$ (3). Take any two points $x$ and $y$ in $X$. Now assume that $x\not\in \overline{\{y\}}$ and $y\not\in \overline{\{x\}}$. Therefore, from our assumption, it follows that $\{x\}\cup (X\setminus\overline{\{x\}})$ and $\{y\}\cup (X\setminus\overline{\{y\}})$ are the unique open neighborhoods of $x$ and $y$ in $\{x\}\cup (X\setminus\overline{\{x\}})$ and $\{y\}\cup (X\setminus\overline{\{y\}})$ respectively. Then $y\in\{x\}\cup (X\setminus\overline{\{x\}})$ and $x\in\{y\}\cup (X\setminus\overline{\{y\}})$, hence it is easily verified that $x\in \overline{\{y\}}$ and $y\in \overline{\{x\}}$, which is a contradiction.
\end{proof}

\begin{corollary}
Let $G$ be a quasitopological group. Then $G$ is dense-ultraconnected if and only if $G$ is a indiscrete space.
\end{corollary}

\begin{proof}
The sufficiency is obvious. Assume $G$ is dense-ultraconnected, then it follows (3) of Theorem~\ref{t22} that, for any two points $x$ and $y$ in $X$, we have $x\in \overline{\{y\}}$ or $y\in \overline{\{x\}}$. Since $G$ is a quasitopological group, we have $y\in \overline{\{x\}}$ for any $x, y\in G$, that is, $\overline{\{x\}}=G$. Therefore, $G$ is a indiscrete space.
\end{proof}

By Theorem~\ref{t22}, the space $G$ in Remark~\ref{rem3} is a dense-ultraconnected paratopological group; it is obvious that $G$ is not a indiscrete space. By Theorem~\ref{t22} again, the following Example $H$ is ultraconnected which is not dense-ultraconnected.

\begin{example}
Let $H=\mathbb{R}$ endowed with a topology $\tau$ which is generated by the following neighborhood base as follows:
for each $x\in\mathbb{R}\setminus\{0, 1\}$, the neighborhood has the form $[x, +\infty)$, and for each $x\in\{0, 1\}$, the neighborhood has the form $\{x\}\cup [n, +\infty), n\in\mathbb{N}$. Clearly, $H$ is ultraconnected and not dense-ultraconnected since $0\not\in\overline{\{1\}}$ and $1\not\in\overline{\{0\}}$.
\end{example}

The following example shows that dense-ultraconnected is not finite productive.

\begin{example}
There exists a dense-ultraconnected paratopological group $G$ such that $G^{2}$ is not dense-ultraconnected.
\end{example}

\begin{proof}
Let $G$ be the paratopological group in Remark~\ref{rem3}. Clearly, $G$ is dense-ultraconnected by Theorem~\ref{t22}. However, $G^{2}$ is not dense-ultraconnected since
$\{(1, 0)\}\not\in\overline{\{(0, 1)\}}$ and $\{(0, 1)\}\not\in\overline{\{(1, 0)\}}$, hence $G^{2}$ is not dense-ultraconnected by Theorem~\ref{t22} again.
\end{proof}

  %%%%%%%%%%%%%%%%%%%%%%%%%%%%


\begin{thebibliography}{99}

\bibitem{AT2008} A. Arhangel'ski\v{\i}, M. Tkachenko,
  {\it Topological groups and related structures}, Atlantis Press, Paris; World
  Scientific Publishing Co. Pte. Ltd., Hackensack, NJ, 2008.

\bibitem{B2022} T. Banakh, Y. StelmakhT, Examples of strongly rigid countable (semi) Hausdorff spaces, arXiv:2211.12579v4.

\bibitem{B1953} M. Brown, A countable connected Hausdorff space, Bull. Amer. Math. Soc., 59 (1953), 367.

\bibitem{CR1982} W.W. Comfort, L.C. Robertson, Proper pseudocompact extensions of compact Abelian group topologies, Proc. Amer. Mathe. Soc., 86(1)(1982), 173--178.

\bibitem{CR1988} W.W. Comfort, L.C. Robertson, Extremal phenomena in certain classes of totally bounded groups, Dissertationes Math. (Rozprawy Mat.), 272(1988), 1--42

\bibitem{CM2007} W.W. Comfort, J. Van Mill J, Extremal pseudocompact abelian groups are compact metrizable, Proc. Amer. Math. Soc., 135(12)(2007), 4039--4044.

\bibitem{CM1989} W.W. Comfort, J. van Mill, Concerning connected, pseudocompact Abelian groups, Topol. Appl., 33(1)(1989), 21--45.

\bibitem{DJ2020} A. Dow, I. Juh\'{a}sz, {\it Dense $k$-separable compacta are densely separable}, Topol. Appl., 283(2020) 107351.

\bibitem{DS2016} D. Dikranjan, D. Shakhmatov, A complete solution of Markov's problem on connected group topologies, Adv. Math., 286(2016), 286--307.

\bibitem{D2019} D. Dikranjan, The gentle, generous giant tampering with dense subgroups of topological groups, Topol. Appl., 259(2019), 6--27.

\bibitem{E1989} R. Engelking, {\it General Topology}, PWN, Warzawa, 1989.

\bibitem{JS1989} I. Juh\'{a}sz, S. Shelah, {\it $\pi(X)=\delta(X)$ for compact $X$}, Topol. Appl., 32(1989) 289--294.

\bibitem{LM1975} R. Levy, R.H. McDowell, {\it Dense subsets of $\beta X$}, Proc. Amer. Math. Soc., 50(1)(1975) 426--430.

\bibitem{LWL2023} F. Lin, Q.Y. Wu, C. Liu, Dense-separable groups and its applications in $d$-independence,  arXiv:2211.14588v3.

\bibitem{M1986} R. Maehara, On a connected dense proper subgroup of $\mathbb{R}^{2}$ whose complement is connected, Proc. Amer. Math. Soc., 1986, 97(3)(1986), 556--558.

\bibitem{W1971} H.J. Wilcox, Dense subgroups of compact groups, Proc. Amer. Math. Soc., 28(2)(1971), 578--580.

\bibitem{WS1976} J.H. Weston, J. Shilleto, {\it Cardinalities of dense sets}, Topol. Appl., 6(1976) 227--240.
  \end{thebibliography}
  \end{document}